\documentclass[12pt]{amsart}
\newcommand{\rank}{\hbox{rank}\ }

\usepackage{amsmath}

\usepackage{amsthm}
\usepackage[dvips]{graphicx}

\usepackage{epsfig}

\usepackage{amssymb}
\usepackage{latexsym}

\newtheorem{Theorem}{Theorem}
\newtheorem{Proposition}{Proposition}
\newtheorem{Lemma}{Lemma}

\theoremstyle{definition}

\def\R{{\mathbb R}}

\def\Z{{\mathbb Z}}

\def\S{{\mathbb S}}

\usepackage[english]{babel}
\usepackage[latin1]{inputenc}
\usepackage{amssymb,amsmath,amsthm,xy,array}
\usepackage[T1]{fontenc}
\usepackage{indentfirst}
\DeclareMathOperator{\Iso}{Iso}

\DeclareMathOperator{\Fix}{Fix}
\DeclareMathOperator{\SO}{SO}
\DeclareMathOperator{\Oo}{O}

\begin{document}

\title[Quotient orbifolds of 3-manifolds of genus two]{On quotient orbifolds  of  hyperbolic 3-manifolds of genus two}

\author[A. Bruno]{Annalisa Bruno}
\address{A. Bruno} 

\email{annalisa.bruno@mail.polimi.it}

\author[M. Mecchia]{Mattia Mecchia*}
\address{M. Mecchia: Dipartimento Di Matematica e Geoscienze, Universit\`{a} degli Studi di Trieste, Via Valerio 12/1, 34127, Trieste, Italy.} \email{mmecchia@units.it}


\thanks{$^{*}$Partially supported by the FRA 2013 grant  ``Geometria e topologia delle variet\`{a}'', Universit\`{a}  di Trieste, and by the  PRIN 2010-2011 grant ``Variet\`{a} reali e complesse: geometria, topologia e analisi armonica''.}

\date{\today}

\begin{abstract}
We analyze the  orbifolds that can be obtained as quotients of   hyperbolic 3-manifolds admitting a Heegaard splitting of  genus two by their orientation preserving isometry groups. The  genus two hyperbolic 3-manifolds are exactly the hyperbolic 2-fold branched coverings of 3-bridge links. If the 3-bridge link is a knot, we prove that the underlying topological space of the quotient orbifold is  either the 3-sphere or a lens space and we describe the combinatorial setting of the singular set for each possible isometry group.  In the  case of 3-bridge links with two or three components, the situation is more complicated and we show that the underlying topological space is the 3-sphere, a lens space or a prism manifold. Finally we present a infinite family of hyperbolic 3-manifolds that are simultaneously the 2-fold branched covering of two inequivalent knot, one with bridge number three and the other one with bridge number strictly greater than three.

\end{abstract}

\maketitle

\section{Introduction}

In this paper we consider quotient orbifolds obtained from the action of finite groups on hyperbolic 3-manifolds admitting a Heegaard splitting of genus two.

\smallskip

A genus $n$ {\sl Heegaard splitting} of a closed orientable 3-manifold $M$ is a
decomposition of $M$ into a union $V_1\cup V_2$ of two handlebodies of genus $n$
intersecting in their common boundary (the Heegaard
surface of the splitting). The {\sl genus of $M$} is the lowest genus for which
$M$ admits a Heegaard splitting.
The only  3-manifold of genus 0 is the 3-sphere while the genus one 3-manifolds are the lens spaces.
We remark that two is the lowest possible genus for a hyperbolic manifold. 

Quotient orbifolds  of 3-manifolds admitting a Heegaard splitting of genus 2   were also studied   by J.Kim  by using different methods (see \cite{K}). In his paper J.Kim considered  only groups leaving invariant the Heegaard splitting of genus 2. Here we do not make this assumption. On the other hand the results in \cite{K} include also  non hyperbolic 3-manifolds.

\smallskip

We recall that 3-bridge knots are strictly related to genus two 3-manifolds.
 An
{\sl $m$-bridge presentation} of a knot $K$ in the 3-sphere $S^3$ is a
decomposition of the pair $(S^3,K)$ into a union
$(B_1,a_1) \cup(B_2,a_2)$ where $B_i$ for $i=1,2$ is a 3-ball and $a_i$ is a
set of $m$ arcs which is trivial in
$B_i$. We shall say that $K$ is an $m$-bridge knot if $m$ is the minimal
number for which $K$ admits an $m$-bridge presentation.

\smallskip

Now a genus two closed orientable surface admits a {\sl hyperelliptic
involution}, (i.e. the quotient of the surface by the involution
is $S^2$): this involution has the property that, up to isotopy, any
homeomorphism of the surface commutes with it.
So for any genus two Heegaard splitting of $M$ there exists an
orientation preserving involution of $M$, which we
shall also call {\sl hyperelliptic}, which leaves invariant the Heegaard
splitting and induces the hyperelliptic
involution on the Heegaard surface (in contrast with the two-dimensional
case, a genus two 3-manifold admits, in
general, more than one hyperelliptic involution, even up to isotopy). The
quotient of $M$ by this involution is
topologically $S^3$ and its singular set is a link $L$.  In this case the Heegaard splitting of $M$ naturally induces a 3-bridge presentation $(B_1,a_1)\cup(B_2,a_2)$ of $L$ where $(B_i,a_i)$ for $i=1,2$ is the quotient of $V_i$.

Conversely, a sphere that induces an $m$-bridge presentation of $L$ lifts to a Heegaard surface of genus $m-1$ of the 2-fold branched covering of $L$.  In particular a 3-bridge presentation  induces a genus two Heegaard splitting of the 2-fold branched covering.

We can conclude that the  hyperbolic 3-manifolds of genus two are exactly  the hyperbolic 3-manifolds that are the  2-fold branched covering of a  3-bridge link. This representation  is not unique, in fact there exist inequivalent 3-bridge knots with the same hyperbolic 2-fold branched covering (for an example see \cite[Section 5]{MR}). In \cite{R} it is proved that a hyperbolic 3-manifold of genus two is the 2-fold branched covering of at most three 3-bridge links. The representation of 3-manifolds as 2-fold branched coverings of knots and links have been extensively studied (see for example the survey by L.Paoluzzi \cite{P} and the recent results by J.E.Greene
 \cite{G})

\smallskip

In this paper we prove the following theorem about the structure of the quotient orbifolds of  hyperbolic 3-manifolds of genus two.
We recall that by the Thurston orbifold geometrization theorem, any periodic  diffeomorphism of a hyperbolic 3-manifold $M$ is conjugate to an isometry of $M$, so we can suppose that the covering transformation of a 3-bridge link is an isometry. 
In the following theorem we consider only link symmetries that preserve the orientation of $S^3$. 

\begin{Theorem}\label{theo:main}
Let $L$ be a 3-bridge link and let $M$ be the 3-manifold of genus 2 that is the 2-fold branched covering of $L$
Suppose that $M$ is hyperbolic  and denote by $G$  the orientation preserving isometry group of $M$.
\begin{enumerate}

\item If $L$ is a knot, then the underlying topological space of $M/G$ is either $S^3$ or a lens space and the combinatorial setting of the singular set of $M/G$ is represented in Figure~\ref{theorem-table}.
If the underlying topological space is a lens space, then  the covering transformation of $L$ is central in $G.$

\item If $L$ has two or three components, the underlying topological space of $G/M$ is $S^3$, a lens space or a prism manifold. If the underlying topological space is a prism manifold, then  

\begin{itemize}

\item
the  group of symmetries of $L$ fixing   setwise each component and preserving the orientation of $S^3$ is a non-trivial  cyclic group, which acts freely on $S^3;$

\item  the symmetries of $L$ preserving the orientation of $S^3$ induce a group acting faithfully on the set of the components of $L$ that is isomorphic to the symmetry group $\mathcal{S}_n$ where $n$ is the number of the components.

\end{itemize}

\end{enumerate}

\end{Theorem}

\begin{figure}[h!]
\begin{center}
\includegraphics[width=9cm]{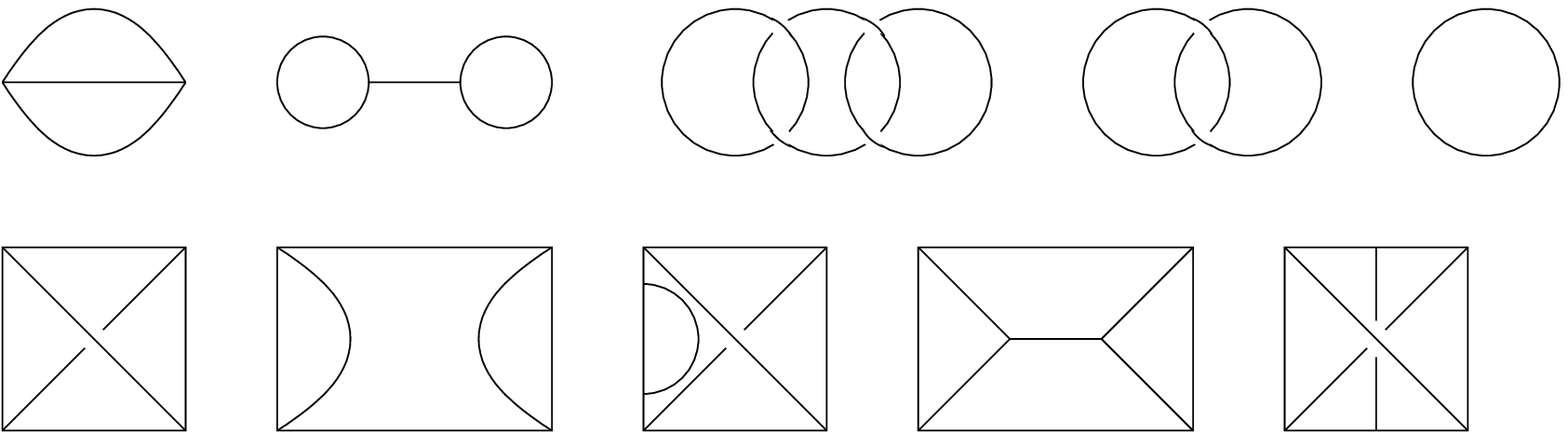}
\caption{Admissible singular sets for $M/G$.}
\label{theorem-table}
\end{center}
\end{figure}

The proof of the theorem is based on the characterization of the isometry group of the hyperbolic 3-manifolds of genus two given in \cite{MR}. 

The underlying topological spaces are analyzed  both for knots and  links with more than one component. In the knot case, if the underlying topological space is not $S^3$, then the hyperelliptic involution of $L$ is central and each element of $G$  projects to a symmetry of $L$
(the 2-fold branched covering  has no `` hidden symmetries'').  We remark also that the situation of quotient orbifolds with underlying topological space that is neither the 3 -sphere nor a lens space is very specific.

More details about the quotients are contained in  Section~\ref{knot} and \ref{link}. In particular in Figure~\ref{final-table}  the case of knots is summarized, distinguishing  for each group, that can occur as $G$, the possible combinatorial settings of the  singular set of $M/G.$
In principle a similar analysis   could be done when $L$ is not a knot  but in the link case the number of possible graphs for the singular set of $M/G$ is very large and we should obtain a very  long and  complicated list.

We know that  the quotient of a genus two 3-manifold by an hyperelliptic involution is an orbifold with underlying topological space $S^3$.  One might ask if each involution that gives $S^3$ as underlying topological space of the quotient  is hyperelliptic. The answer is no, in fact in the last section of the paper we present an infinite family  of  genus two 3-manifolds that are the 2-fold branched coverings of  knots with bridge number strictly greater than four. Since the bridge number of the knots is not three, the covering involutions of these branched coverings are   not  hyperelliptic.   This family gives  also examples of 2-fold branched coverings  where a sphere giving a minimal bridge presentation  of the knot  does not lift to a Heegaard surface of minimal  genus. A different method  to obtain examples of this phenomenon  and some comments about it can be found in [J].


\section{Preliminaries}\label{prel}

In this section we present some preliminary results about finite group action on 3-manifolds, and in particular on the 3-sphere.

\begin{Proposition} \label{prop:normalizer}
Let $G$ be a finite group of orientation preserving diffeomorphisms
of a closed orientable 3-manifold and let $h$ be an element in $G$ with nonempty
connected fixed point set. Then the normalizer $N_G h$ of the subgroup generated
by $h$ in $G$ is isomorphic to a subgroup of a semidirect product
$\Z_2 \ltimes  (\Z_a \times \Z_b)$,
for some positive integers $a$ and $b$, where a generator of $\Z_2$ (an $h$-reflection)
acts on the normal subgroup $\Z_a \times  \Z_b$ of $h$-rotations by sending each element to
its inverse. 
\end{Proposition}

\begin{proof}
The fixed point set of $h$ is a simple closed curve $K$, which is invariant under the
action of $N_Gh$. By a result of Newman (see~\cite[Theorem 9.5]{Br}), a periodic transformations of a manifold which is the identity on an open subset is the identity. Thus the
action of an element of $N_Gh$ is determined by its action on a regular neighborhood of $K$ where it is a standard action on a solid torus. Every element of $N_Gh$ restricts to a reflection (strong inversion) or to a (possibly trivial) rotation on $K$. 
The subgroup of $h$-rotations has index one or two in $N_Gh$ and is abelian.
It has a cyclic subgroup (the elements fixing $K$ pointwise) with cyclic quotient
group, so it is abelian of rank at most two.
\end{proof}

We consider now the finite subgroups of $\SO(4)$ and their action on the unit
sphere $S^3 = \{(x, y, z,w)\in \R^4|x^2 + y^2 + z^2 + w^2=1\}.$ 
We recall that a non-trivial element of prime order in $\SO(4)$ either acts freely or fixes pointwise a simple closed curve in $S^3$. The finite subgroups of $\SO(4)$ are classified by Seifert and Threlfall (\cite{TSe1} and \cite{TSe2}). In Lemma~\ref{lemma:orth-group}  we collect some properties of these groups which we need in this paper. The results in Lemma~\ref{lemma:orth-group} can be obtained from the classification and the results contained in~\cite{MS}, but here a direct proof seems to be more suitable. Point 2. of Lemma~\ref{lemma:orth-group} is taken from  \cite{BBW}, that was  available at http://www.tufts.edu/$\sim$gwalsh01; the paper is no longer available,  since it was replaced by \cite{BBCW}, where this statement is not considered.

\begin{Lemma}\label{lemma:orth-group}
Let $G$ be a finite subgroup of $\SO(4).$

\begin{enumerate}
\item Suppose that $G$ is abelian, then either it has rank at most two or it is an
elementary 2-group of rank three. If $G$ acts freely on $S^3$, it is cyclic. If $G$ has
rank at most two, then either at most two simple closed curves of $S^3$ are fixed
pointwise by some nontrivial element of $G$ or $G$ is an elementary 2-group
of rank two and the whole $G$ fixes two points (where the fixed-point sets of the three involutions meet).
\item If $G$ is generalized dihedral (i.e. $G$ is a semidirect product of an  abelian
subgroup of index two with a subgroup of order whose  generator
acts dihedrally on the abelian subgroup of index two), then the underlying
topological space of $S^3/G$ is $S^3.$
\item If $G$ has a cyclic normal subgroup $H$ such that $G/H$ is cyclic of odd order, then
$G$ is abelian.
\end{enumerate}
\end{Lemma}

\begin{proof}

1) Since G is abelian, the elements of the group can be simultaneously
conjugate to block-diagonal matrices, i.e. $G$ can be conjugate to a group such
that each element has the following form:

\begin{eqnarray*}
\left(
\begin{array}{c|c}
A & 0_2
\\
\hline
0_2 & B
\end{array}
\right)
\end{eqnarray*}

where $A,B \in \Oo(2)$ and $0_2$ is the $2\times 2$ matrix whose entries are all zero. Then,
by using  standard arguments from linear algebra,  Point 1 can be proved.

\medskip

2) In this case the group $G$ has an abelian subgroup $R$ of rank at most two
and index two. Let   $\phi$ and $\psi$ be two elements in $\SO(4)$ that generate $R$.
Since $R$ is abelian, the action of $R$ leaves setwise invariant a 2-dimensional
plane $P$ in $\R^4$, which corresponds to a simple closed curve in $S^3$. The group $R$
leaves invariant also $P^{\perp}$, and $P \oplus P^{\perp}$ is a $R$-invariant decomposition of $\R^4.$ We
will prove that also the action of $G$ on $\R^4$ leaves invariant setwise a 2-dimensional
plane, that may be different from $P$.
Let $\sigma$ be  an involution not in $R$. We have: $\psi (\sigma (P))= \sigma (\psi^{-1} (P))= \sigma (P)$ and $\phi (\sigma (P))= \sigma (\phi^{-1} (P))= \sigma (P)$, this implies that the set $Q=P \cap \sigma (P)$ is $G$-invariant.

If $Q=P$, then $P$ is a 2-dimensional plane left setwise invariant by $P.$

If $Q$ is a subspace of dimension 1, then we can construct explicitly another plane that is $G$-invariant.
Let $\{ v,w \}$ be an orthonormal basis of $P$ such that $v \in Q$.
Since $\psi (v)= \pm v$ and $\phi(v)= \pm v$ and $P$ is both $\psi$- and $\phi$-invariant, we must have that  $\psi(w)= \pm w$ and $\phi(w)= \pm w$.
Therefore the plane spanned by the vectors $v$ and $w + \sigma (w)$ is $G$-invariant.

If $Q=\{ 0 \}$, we fix again an orthonormal basis $\{ v,w \}$ of $P$.
If $\psi$ or $\phi$ acts as a reflection on $P$, then a normal subgroup of $G$ leaves pointwise invariant a 2-dimensional plane which is left setwise invariant by $G.$
We can suppose that  $\psi$ and $\phi$ act as rotations on the plane $P$ and we will prove that  the plane spanned by the couple of vectors $v + \sigma (v)$ and $w - \sigma(w)$ is a $G$-invariant plane.
In fact we have that $\psi (v + \sigma (v))= \psi (v) + \psi(\sigma (v))= \psi (v) + \sigma (\psi^{-1}(v))$.
Supposing $\psi$ acting on the basis in the following way: $\psi (v)= \alpha v + \beta w$, $\psi(w)= - \beta v + \alpha w$, we have that $\psi^{-1}(v)= \alpha v - \beta w$ and $\psi^{-1}(w)= \beta v + \alpha w$.
Then $\psi (v + \sigma (v))= \alpha v + \beta w + \sigma (\alpha v - \beta w)= \alpha (v + \sigma(v)) + \beta (w - \sigma (w))$ and $\psi(w - \sigma (w))= - \beta v + \alpha w - \sigma (\beta v + \alpha w)= - \beta (v + \sigma (v)) + \alpha (w - \sigma(w))$.
The same argument works with $\phi$, since it is a rotation on $P$ too, moreover $\sigma (v + \sigma (v))= v + \sigma (v)$ and $\sigma (w - \sigma(w))= - (w - \sigma(w))$.
This completes the proof of the fact that $G$ leaves invariant a 2- plane in $\R^4$.

At this point we can suppose that $\sigma(x,y,z,w)= (x,-y,z,-w)$, up to conjugacy.
The whole isometry group G respects the Heegaard splitting $S^3 = T_1 \cup T_2$, where $T_1 = \{(x, y,z,w) \in S_3 : x^2+y^2 \geq 1/2\}$ and $T_2 = \{(x, y,z,w)\in S^3 :
x^2+y^2 \leq 1/2\}$. We obtain that $G$ acts on the solid tori $T_1$ and $T_2$ in such a way
that their quotients by $G$ are two solid balls $B_1$ and $B_2$; then $S_3/G$ is given by the
gluing of a couple of solid balls, that is known to be a 3-sphere $S_3.$

3) Let $\psi$ be a generator of $H$ and $\sigma$ be an element of $G$ such that $\sigma H$ is a generator of $G/H.$
We denote by $h$ the natural number  smaller then the order of $\psi$ such that $\sigma\psi \sigma^{-1}=\psi^h$

If $\psi$ (or any nontrivial element of $H$) fixes pointwise a 2-dimensional plane $P$ in $\R^4$, then $\sigma$ fixes setwise the same 2-dimensional plane  and $P\oplus P^{\perp}$ is a $G$-invariant decomposition of $\R^4.$ In this case we  can reduce the problem to the analysis of the finite subgroups of $\Oo(2)$ and we are done.

We know that $\psi$ can be conjugate by a matrix in $\SO(4)$ in the form:

\begin{eqnarray*}
\left(
\begin{array}{c|c}
A & 0_2
\\
\hline
0_2 & B
\end{array}
\right).
\end{eqnarray*}

If $A$ and $B$ have different orders, then  we obtain in $H$ a nontrivial element fixing pointwise a 2-dimensional plane, hence we get the thesis.

We can suppose that $A$ and $B$ have the same order.

We consider $\psi$ as a complex matrix; if  $v$ is an eigenvector of $\phi$ corresponding to the eigenvalues $\lambda$, then $\sigma(v)$ is an eigenvector for $\phi$ corresponding to the eigenvalues $\lambda^{h}.$ 
 
Suppose first that  $\lambda^{h}=\lambda$ for an eigenvalue $\lambda$; since the multiplicative order of $\lambda$ equals the order of $\psi$, we obtain $\sigma\psi \sigma^{-1}=\psi$  and we get the thesis.

Then we can suppose that $\lambda^{h} \neq \lambda$ for each eigenvalue $\lambda$; $\sigma$ induces a bijection on the set of eigenvalues that does not fix  any of them.
If one of the eigenvalues is -1, then the order of $\psi$ is two and the matrix is diagonal ($A$ and $B$ have the same order); in this case $\psi$ is central in $G.$
Therefore $\psi$ has two or four different eigenvalues, in any case $\sigma^4$ leaves invariant each eigenvalue, hence $\sigma^4$ commutes with $\psi$.
Since the order of $\sigma H$ is odd,  $\sigma^4 H$ generates $G/H$ and we obtained that $G$ is an abelian group.

\end{proof}

\section{The knot case}\label{knot}

In this section we prove Theorem~\ref{theo:main} in the knot case.

We recall that,  since $L$ is a knot,  $M$  is a  $\Z_2$-homology sphere and, by Smith theory,  the fixed point set of an involution acting on $M$ is either  empty or a simple closed curve. 

The method we use to investigate $M/G$  is to pass through iterated quotients using a subnormal series of  subgroups of $G$.
This method can be applied thanks to the fact that, if $G$ is a group acting on a manifold $M$ and $H$ is a normal subgroup of $G$, then the action of $G$ induces an action of $G/H$ on the quotient $M/H.$

In~\cite{MR} it is proved that either   there exists a hyperelliptic involution central in $G$  or $G$ is isomorphic to a subgroup of $\Z_2\times \S_4$. We consider the two cases.

\medskip

\noindent
\textbf{Case 1: $G$ contains  a central hyperelliptic involution.}

\smallskip

Let $t$ be a hyperelliptic involution contained in the center of $G$ and we suppose that  $L$ is  the hyperbolic 3-bridge knot that is the projection of the fixed point set of $t$ on $M/\langle t \rangle\cong S^3$.

The whole group $G$ projects to $M/\langle t \rangle$ and the quotient $M/G$ can be factorized through $(M/\langle t \rangle)/(G/\langle t \rangle)$.  The Thurston orbifold geometrization theorem (see~\cite{BLP}) and  the spherical space form  conjecture for free actions on $S^3$  proved by Perelmann (see~\cite{MoT}) imply that every finite group of diffeomorphism of the 3-sphere is conjugate to a finite subgroup of $\textrm{SO}(4)$. The finite subgroups of  $\textrm{SO}(4)$ were classified by Seifert and Threlfall (\cite{TSe1} and \cite{TSe2}, more details can be found in~\cite{MS}). Thus we can suppose that $G/\langle t \rangle$ is a group of isometries of $S^3$  leaving  setwise invariant the knot $L.$ We remark that, since $L$ is not a trivial knot, by the positive solution of Smith conjecture, $G/\langle t \rangle$ acts faithfully on $L$. In particular  $G/\langle t \rangle$ is cyclic or dihedral (see Proposition~\ref{prop:normalizer}). 

If   the fixed point set (that may be empty) of a symmetry of $L$  is disjoint from the knot, we call it {\it a $L$-rotation}. If the fixed point set of a symmetry of $L$ intersects the knot in two points, we call it {\it a $L$-reflection}.

Suppose first that $G/\langle t \rangle$ consists only of $L$-rotations.
In this case  $G/\langle t \rangle$ is cyclic and there are at most two simple closed curves that are fixed pointwise by some nontrivial element of $G/\langle t \rangle$ (see Lemma~\ref{lemma:orth-group}), thus the  singular set of $M/G$ is a link with at most three components.
We recall that the quotient of  a 3-sphere by an isometry with non-empty fixed point set is a 3-orbifold with the 3-sphere as  underlying topological space and a trivial knot as singular set, while the quotient by an isometry acting freely is a lens space. 
Therefore if  $G/\langle t \rangle$ is generated by elements with nonempty fixed point set the underlying topological space of  the orbifold $M/G$ is  a  3-sphere, otherwise it is a lens-space.

If the  group $G/\langle t \rangle$ contains a reflection of $L$, then $G/\langle t \rangle$  is either dihedral or isomorphic to $\mathbb{Z}_2$.
In any case, by Lemma~\ref{lemma:orth-group}, the underlying topological space of the quotient is a 3-sphere.

If $G/\langle t \rangle \cong \mathbb{Z}_2$, then the orbifold $M/G$ can be obtained as the quotient of $M/\langle t \rangle$ by a $L$-reflection; so the singular set of  $M/G$ is a theta-curve.
For what concerns the dihedral case, first of all we recall that for $G/\langle t \rangle$ to be dihedral means that it is generated, up to conjugacy, by a $L$-reflection $s$ and by a $L$-rotation $r$.
In Lemma~\ref{lemma:orth-group} we  defined two tori $T_1$ and $T_2$ such that $T_1\cup T_2=S^3$ that are left invariant by $G$. Referring to the notation of the proof of Lemma~\ref{lemma:orth-group}, we define   $C_1=\{(x,y,z,w)\in S^3|x^2+y^2=0\}$ and $C_2=\{(x,y,z,w)\in S^3|x^2+y^2=1\}$; these curves are  the cores of the tori $T_1$ and $T_2.$ We can suppose by conjugacy that the fixed-point sets  of the $L$-rotations are contained in $C_1\cup C_2.$

We consider then  the singular set of $S^3/(G/\langle t \rangle)$ (where the knot $L$ is not considered singular).
The singular set of $S^3/(G/\langle t \rangle)$ is contained in the union of the projection of $C_1 \cup C_2$ with the projection of  the fixed point sets of the $L$-reflections.

Let $n$ be the order of the $L$-rotation $r.$
We distinguish two cases: $n$ odd or $n$ even.

\bigskip

\begin{figure}[t]
\begin{center}
\includegraphics[width=8cm]{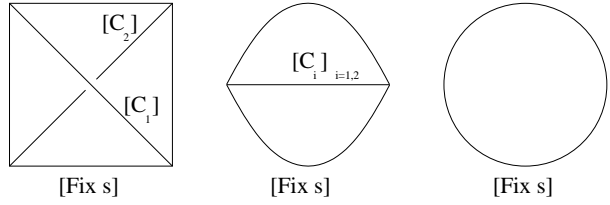}
\caption{Possible singular sets of $S^{3}/(G/\langle t \rangle)$  if $G/\langle t \rangle$ is dihedral and $n$ odd.}
\label{fig17}
\end{center}

\end{figure}

If $n$ is odd, then all the $L$-reflections are conjugate. 
Therefore if we consider the fixed point sets of  the  $L$-reflections, the projections of these fixed point sets are all identified in a unique closed curve in the quotient  $S^{3}/\langle r \rangle.$ The involution $s$ acts as a reflection also on the curves $C_{1}$ and $C_{2}$. If we consider the action of the projection of $s$ to $S^{3}/\langle r \rangle$, we can describe the possible combinatorial settings for the singular set. These are represented in Figure~\ref{fig17}. Notice that the singular set of $S^{3}/\langle r \rangle$ can also be empty or have only one component; the number of components of  $S^{3}/\langle r \rangle$  depends on the number of the simple closed curves, that are fixed pointwise by any $L$-rotation. We denote by $[C_1]$, $[C_2]$  and $[\Fix s]$ the projections   to $S^3$ of $C_1$, $C_2$ and  $\Fix s$, respectively.

However here there is something to remark.
The first graph is only one of the two combinatorial settings that can be built  with a closed curve and two edges with different endpoints.
The second possibility is the graph in Figure \ref{fig18}. By Lemma~\ref{lemma:orth-group} we can choose up to conjugacy $s:S^3\subset \R^4\rightarrow S^3$ as the map sending $(x,y,z,w)$ to $(x,-y,z,-w)$, and it is easy to see that the fixed point set of $r$  meets alternately $C_1$ and $C_2$, so this graph does not occur.

\begin{figure}[htb]
\begin{center}
\includegraphics[width=2cm]{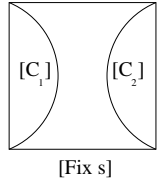}
\caption{Another combinatorial setting build with a closed curve and two edges.}
\label{fig18}
\end{center}
\end{figure}

To obtain the singular set of $M/G$,  to each  graph in Figure~\ref{fig17} we add $[L]$, where $[L]$ is the projection of $L$ to $S^3/(G/\langle t \rangle)$.
Since $s$ is a $L$-reflection, $[L]$ is an edge with  endpoints contained in $[\Fix s].$
Figure~\ref{fig21} contains all the possibilities, up to knotting; all the edges, except $[C_1]$ and $[C_2]$, must have singularity index two.

\begin{figure}[h]
\begin{center}
\includegraphics[width=7cm]{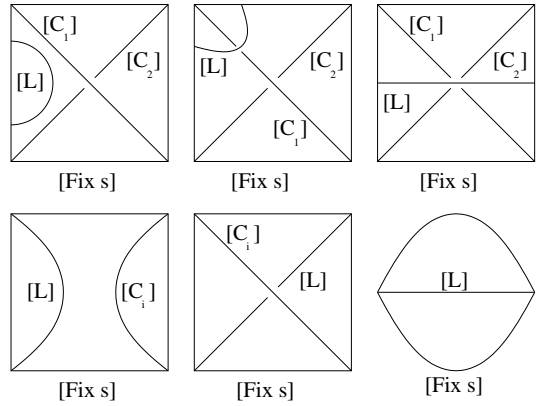}
\caption{Possible singular sets of $M/G$  when $G/\langle t \rangle$ is dihedral of order $2n$ with $n$ odd.}
\label{fig21}
\end{center}
\end{figure}

\begin{figure}[t]
\begin{center}
\includegraphics[width=11cm]{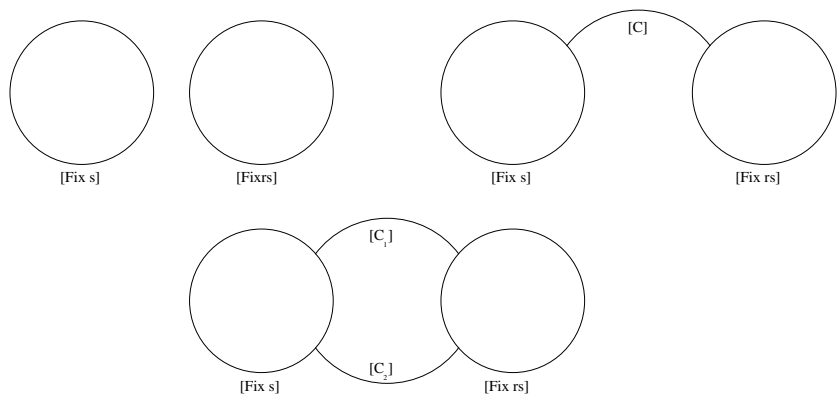}
\caption{Possible singular sets of $S^3/(G/\langle t \rangle)$  when $G/\langle t \rangle$ is dihedral of order $2n$ with $n$ even and the central involution acts freely.}
\label{fig19}
\end{center}
\end{figure}

On the other hand if $n$ is  even, then we do not  have a unique conjugacy class for all the $L$-reflections  of $G/\langle t \rangle$.
Since the fixed point sets of all the elements in the  same conjugacy class project to a single curve in the quotient $M/(G/\langle t \rangle)$, we take into consideration from now on only $\Fix s$ and $\Fix rs$, taking one representative element for each conjugacy class.
In this case the fixed point sets of the $L$-reflections are not all identified in the quotient, but are collected into two different subsets of the singular set of  $S^3/(G/\langle t \rangle)$, that we can denote simply by $[\Fix s]$ and $[\Fix rs]$.

Notice also that if $n$ is even, $r^{n/2}$ is a central element in $G/\langle t \rangle$,  hence we have a $L$-rotation fixing  setwise $\Fix s$ and $\Fix rs$.

The type of action of $r^{n/2}$ on   $\Fix s$ and $\Fix rs$ influences how these curves project to  $S^3/\langle r \rangle.$ In fact, according if this element has empty or non-empty fixed point set, different situations occur.

\begin{figure}[b]
\begin{center}
\includegraphics[width=12cm]{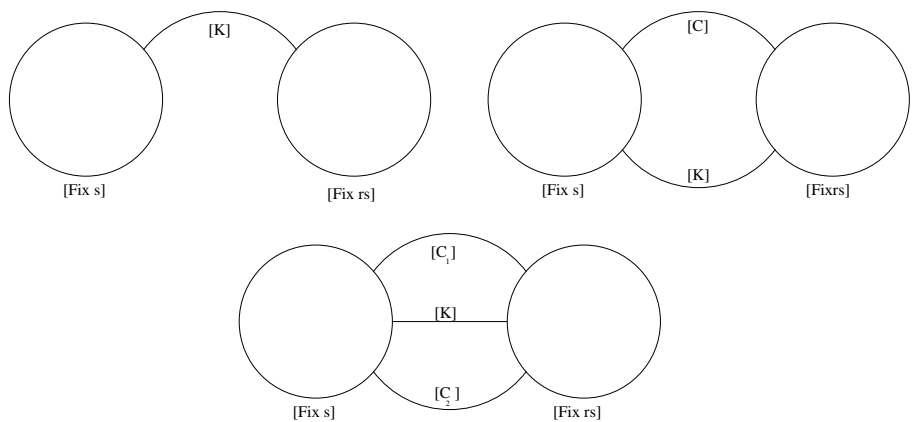}
\caption{Possible singular sets of $M/G$ when $G/\langle t \rangle$ is dihedral of order $2n$ with $n$ even and the central involution acts freely.}
\label{fig22}
\end{center}
\end{figure}

If $\Fix r^{n/2}$ is empty, then it acts on  $\Fix s$ and $\Fix rs$  as a rotation, and the  fixed point sets of $s$ and $rs$ project to two distinct closed curves in $S^{3}/\langle r \rangle$. We remark that $S^{3}/\langle r \rangle$ is  not a 3-sphere. 
The projection of $s$ to $S^{3}/\langle r \rangle$ is an involution which acts as a reflection on the projections of $C_1$ and $C_2$ and  such that its fixed point set consists of the projections of $\Fix s$ and $\Fix rs$. The possible combinatorial structures of the singular set of $S^3/(G/\langle t \rangle)$ are presented in Figure~\ref{fig19}.

We recall that $L$ meets both $\Fix s$ and $\Fix sr$ and for the singular set of $M/G$ we obtain one possibility for each of the graphs, as shown in Figure \ref{fig22} (all the edges but  $[C_1]$ and $[C_2]$ must have singularity index two).

If $\Fix r^{\frac{n}{2}}$ is non-empty, then clearly it coincides either with $C_{1}$ or with $C_{2}$.
In this case, since $r^{\frac{n}{2}}$ commutes both with  $s$ and with $rs$ and also with any other involution of $G/\langle t \rangle$, we obtain that $r^{n/2}$  acts as a strong inversion on both of the closed curves $\Fix s$ and $\Fix rs$ (see Lemma~\ref{lemma:orth-group}).
Therefore the projections of  $\Fix s$ and $\Fix rs$ are two arcs in $S^3/\langle r \rangle$ with both endpoints in common.
Moreover the endpoints of $[\Fix s]$ and $[\Fix rs]$ in $S^3/(G/\langle t \rangle)$ coincide with the endpoints of the 
arc given by the projection of $\Fix r^{\frac{n}{2}}$.
If $C_{1}$ is  $\Fix r^{\frac{n}{2}}$, as before $C_{2}$ links $[\Fix s]$ and $[\Fix rs]$; the roles of $C_1$ and $C_2$ can be exchanged. 
So the possible settings for the singular set of $S^3/(G/\langle t \rangle)$ are the ones represented in Figure~\ref{fig20}.

From these we can build three different graphs that, up to knottings,  are possible singular sets of $M/G$, two from the first graph of Figure~\ref{fig20} and one from the second.
The admissible results are shown in Figure \ref{fig23} (again all the edges, except $[C_1]$ and $[C_2]$, have singularity index two).

\begin{figure}[htb]
\begin{center}
\includegraphics[width=4cm]{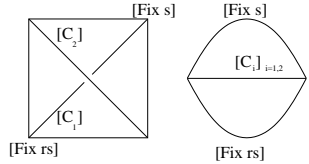}
\caption{Possible singular sets of $S^3/(G/\langle t \rangle)$  when $G/\langle t \rangle$ is dihedral of order $2n$ with $n$ even and the central involution does not act freely.}
\label{fig20}
\end{center}
\end{figure}

\begin{figure}[htb]
\begin{center}
\includegraphics[width=7cm]{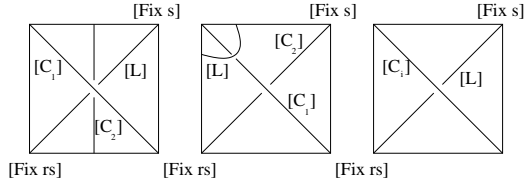}
\caption{Possible singular sets of $M/G$  when $G/\langle t \rangle$ is dihedral of order $2n$ with $n$ even and the central involution does not act freely.}
\label{fig23}
\end{center}
\end{figure}

\medskip

\noindent
\textbf{Case 2: $G$ is isomorphic to a subgroup of $\Z_2\times \S_4.$}
\smallskip

By \cite{R} the number of hyperelliptic involution is at most three.
We recall that by \cite{BS} and \cite{RuS} two hyperelliptic involutions commute and their fixed point sets meet in two points.
Here we consider  groups containing a non-central hyperelliptic involution.
In this case $G$ contains a conjugacy class of  hyperelliptic involutions with  two or three elements (the property to be hyperelliptic is invariant under conjugation).
These groups are described in the proof of \cite[Theorem 1]{MR} (case c) and d) - pages 7 and 8.)  

\medskip

\noindent
\textbf{Case 2.1: $G\cong \mathbb{D}_{8}$.}
\smallskip

This case occurs if we have a conjugacy class of hyperelliptic involutions with two elements which we denote by  $t_1$ and $t_2.$
By the properties of hyperelliptic involutions, $t_1$ and $t_2$ generate an elementary subgroup of rank 2 in $G$, and we have a subnormal series, $$ \langle t_{1} \rangle \lhd \langle t_{1},t_{2} \rangle \lhd G.$$

The orbifold  $M/\langle t_{1} \rangle $ has $S^3$ as underlying topological space and a knot as singular set.
We consider now  $M/\langle t_{1},t_{2} \rangle$ which is diffeomorphic to the quotient of $(M/\langle t_{1} \rangle)$ by the projection of $\langle t_{1},t_{2}\rangle$ to $M/\langle t_{1} \rangle$. Since $t_2$ has non empty fixed point set  and is a $\Fix t_1$-reflection, we obtain that the underlying topological space of $M/\langle t_{1},t_{2} \rangle$ is $S^3$ and its singular set is a knotted theta curve.

\begin{figure}[h]
\begin{center}
\includegraphics[height=1.8cm]{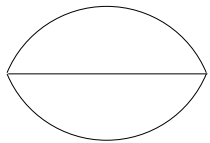}
\caption{A theta curve.}
\label{fig1}
\end{center}
\end{figure}

Now we consider the action of $G/\langle t_{1},t_{2} \rangle$ on $M/\langle t_{1},t_{2} \rangle.$
For a period two action on a  theta-curve $\theta$ we have only three possibilities that are represented in Figure~\ref{fig2}.

\begin{figure}[h]
\begin{center}
\includegraphics[width=11cm]{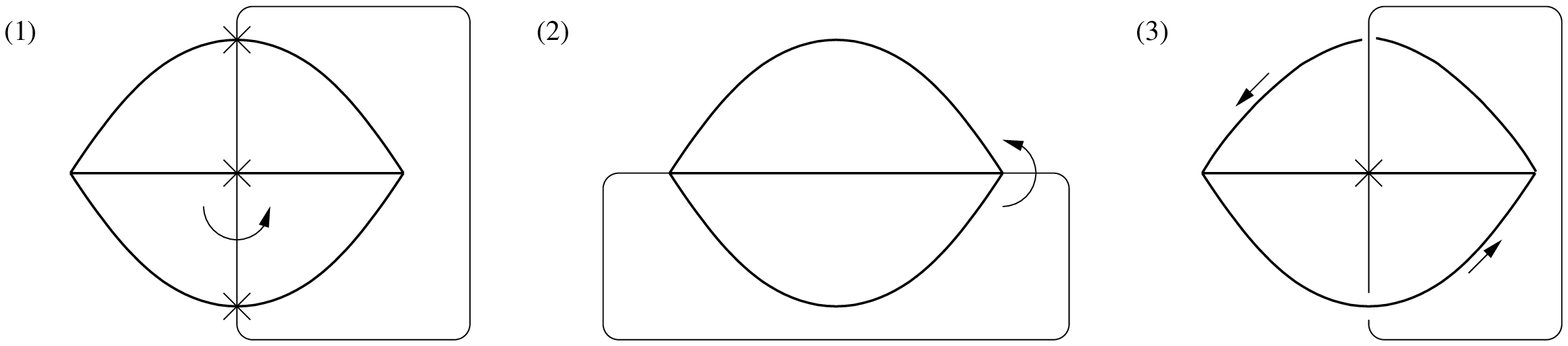}
\caption{Possible actions with period two on a theta curve.}
\label{fig2}
\end{center}
\end{figure}

We can make some remarks on the actions represented.
The first action fixes all the three edges and interchanges the vertices, therefore it acts as a rotation with period two around an axis that intersects all the three edges of the theta-curve and leaves the fixed point sets of $t_{1}$ and $t_{2}$ invariant; the second one acts as a rotation of order two around an axis that contains one of the edges of the theta-curve, therefore it fixes one entire edge and the vertices and interchanges the remaining two edges; the third one acts as a rotation again of order two, but this time the axis intersects only one of the edges and in only one point, therefore it fixes only the intersection point of the theta-curve with the axis, leaves setwise invariant the edge that intersects the axis and interchanges the other two edges and the vertices.
Since we already know that $\mathbb{D}_{8}$ interchanges the fixed point sets of $t_{1}$ and $t_{2}$, the first action is obviously not possible.
Since the non trivial element of $G/\langle t_{1},t_{2} \rangle$ has non-empty fixed point set, the orbifold $M/G$ has  $S^{3}$ as underlying topological space; the  singular set is, up to knottings, one of the two graphs represented in Figure \ref{fig3} (a theta curve and a "pince-nez" graph).
If we obtain a theta curve, then one of the edge has singularity index four (in this case the elements of order four in $G$ have nonempty fixed point set).

\begin{figure}[h]
\begin{center}
\includegraphics[width=10cm]{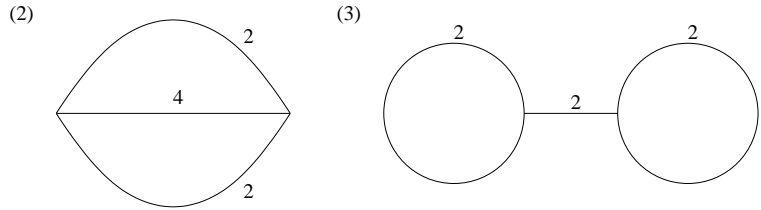}
\caption{A theta-curve and a "pince-nez" graph.}
\label{fig3}
\end{center}
\end{figure}

\medskip

\noindent
\textbf{Case 2.2: $G \cong \mathbb{Z}_{2} \times \mathbb{D}_{8}$.}
\smallskip

This is the second group occurring in  the proof of \cite[Theorem~1]{MR}, when  the existence of a conjugacy class of hyperelliptic involutions with two elements is assumed;  we denote again the two hyperelliptic involutions by  $t_1$ and $t_2.$

The first two quotients we consider are the same of the preceding case and we obtain that $M/\langle t_{1},t_{2} \rangle$ is known.

Let $A$ be the subgroup of $\Iso ^{+} M$ obtained by extending $\langle t_{1},t_{2} \rangle$ by a non-trivial element of the centre of $G$ ($t_1$ and $t_2$ cannot be in the center of $G$).
This means $A \cong \mathbb{Z}_{2} \times \mathbb{Z}_{2} \times \mathbb{Z}_{2}$ and $A \lhd \Iso ^{+} M$.
The subnormal series  we consider in this case is the following: $$ \langle t_{1} \rangle \lhd \langle t_{1},t_{2} \rangle \lhd A \lhd G$$
We consider now the projection of the action of $A$ on $M/\langle t_{1},t_{2} \rangle$.
$A$ acts leaving both  hyperelliptic involutions $t_{1}$ and $t_{2}$ fixed.
This means that it does not interchange the fixed point sets of $t_{1}$ and $t_{2}$.
This time the only possible action of the three represented in Figure \ref{fig2} is the first and the resulting singular set of $M/A$ can be represented, up to knottings, as in Figure~\ref{fig4}, by a tetrahedral graph. Since the action of $A/\langle t_{1},t_{2} \rangle$ is not free, the underlying topological space of $M/A$ is $S^3.$

\begin{figure}[htb]
\begin{center}
\includegraphics[width=2.5cm]{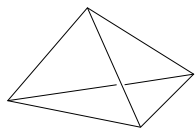}
\caption{Singular set of $M/A$: a tetrahedral graph.}
\label{fig4}
\end{center}
\end{figure}

The last extension to take into consideration is $A \lhd G$, in particular we consider the action of $G/A$ on $M/A.$
We ask what  actions of period two are combinatorially admissible  on a tetrahedral graph.
These are represented in Figure~\ref{fig5}. 

\begin{figure}[htb]
\begin{center}
\includegraphics[width=8cm]{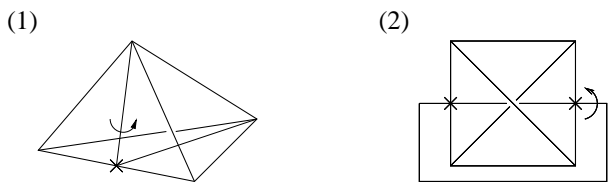}
\caption{Possible actions of period two on a tetrahedral graph.}
\label{fig5}
\end{center}
\end{figure}

The actions represented are respectively a rotation around an axis containing  one of the edges and meeting in a point the opposite one  (1) and a rotation around an axis meeting a couple of non adjacent edges in one point (2). Therefore we obtain that $M/G$ is an orbifold  with underlying topological space a sphere $S^{3}$ and with two possible singular sets, that are the graphs represented in Figure~\ref{fig6} (always up to possible knottings).

\begin{figure}[h]
\begin{center}
\includegraphics[width=7cm]{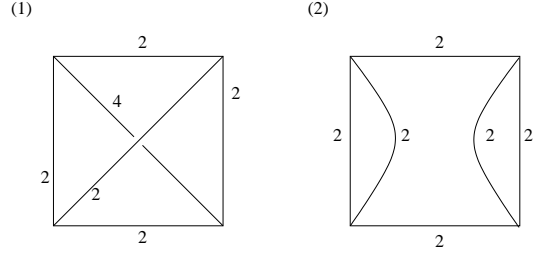}
\caption{Graphs that can occur as singular sets of $M/G$, when $G\cong \mathbb{Z}_{2} \times \mathbb{D}_{8}$.}
\label{fig6}
\end{center} 
\end{figure}

\medskip

\noindent
\textbf{Case 2.3: $G \cong \mathbb{A}_{4}$.}
\smallskip

The case $G \cong \mathbb{A}_{4}$ is the first we encounter in which $M$ admits a conjugacy class of three hyperelliptic involutions.
This condition is satisfied in all the remaining cases  (see proof of \cite[Theorem 1]{MR}) and we will denote   the three hyperelliptic involutions by $t_{1}$, $t_{2}$ and $t_{3}$.
Just as before, we consider a subnormal series of  subgroups of $\mathbb{A}_{4}$: $$\langle t_{1} \rangle \unlhd \langle t_{1},t_{2} \rangle \unlhd \Iso ^{+} M.$$
The first two quotients we need to perform are the same  encountered in the previous cases, therefore we begin analyzing the projection of the action of $G$ on the last quotient $M/\langle t_{1},t_{2} \rangle$, that we recall is an orbifold  with underlying topological space $S^{3}$ and singular set a theta-curve.
Noticing that the index of $\langle t_{1},t_{2} \rangle$ in $G$ is three, it follows that $G/\langle t_{1},t_{2} \rangle$ acts faithfully as a rotation with period three on $M/\langle t_{1},t_{2} \rangle$.
There is only one action of this type, that is a rotation around an axis that passes through the vertices of the theta-curve; the rotation permutes the three edges cyclically.

The action of $\frac{\Iso ^{+} M}{\langle t_{1},t_{2} \rangle}$ on $\frac{M}{\langle t_{1},t_{2} \rangle}$ is clearly not free, so the underlying topological space of $M/G$  is necessarily $S^{3}$.
Moreover we can notice that the result of this action is  again a theta-curve, but with different orders of singularity of the edges as shown in Figure~\ref{fig8}.

\begin{figure}[h]
\begin{center}
\includegraphics[width=2cm]{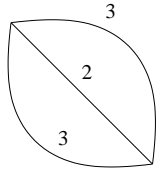}
\caption{Singular set of $M/G$, when $G \cong \mathbb{A}_{4}$.}
\label{fig8}
\end{center}
\end{figure}

\medskip

\noindent
\textbf{Case 2.4: $G \cong \mathbb{Z}_{2} \times\mathbb{A}_{4}$.}
\smallskip

The subnormal series of  subgroups we use this time is the following: $$ \langle t_{1} \rangle \lhd \langle t_{1},t_{2} \rangle \lhd A \lhd G,$$
where $A$ is, as before, the normal subgroup of $G$ isomorphic to $\mathbb{Z}_{2} \times \mathbb{Z}_{2} \times \mathbb{Z}_{2}$ obtained extending $\langle t_{1},t_{2} \rangle$ by an element of the centre of the group $G$.
In light of what we saw in Case 2.2, we already know that $M/A$ is an orbifold with underlying topological space the 3-sphere $S^{3}$ and singular set a tetrahedral graph.
Therefore we can analyze directly the  action of $G/A$ on $M/A$. The group $G/A$ has order three and  there is only one admissible action of order three on a tetrahedral graph  that  is a rotation around an axis passing through one of the vertices of the graph that  permutes  cyclically the three edges containing the vertex fixed by the action, as well the three edges not containing the vertex.
The singular set of $M/G$ is shown, up to knottings, in Figure~\ref{fig10}.

\begin{figure}[ht]
\begin{center}
\includegraphics[height=2cm]{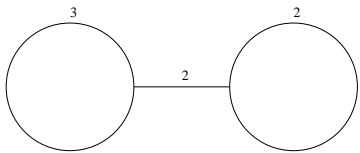}
\caption{Singular set of $M/G$, when $G \cong \mathbb{Z}_{2} \times \mathbb{A}_{4}$.}
\label{fig10}
\end{center}
\end{figure}

Notice that here too the orders of singularity of the edges are different.
Since the action of $G/A$  is not free, the underlying topological space of $M/G$ is $S^{3}.$

\medskip

\noindent
\textbf{Case 2.5: $G \cong \mathcal{S}_{4}$.}
\smallskip

Since $t_1,$ $t_2$ and $t_3$ are conjugate and commute, the subgroup $H$ of $G$ isomorphic to $\mathbb{A}_{4}$ contains the three hyperelliptic involutions.
To study $M/G$ we consider the following subnormal series of subgroups: $$ \langle t_{1} \rangle \lhd \langle t_{1},t_{2} \rangle \lhd H \lhd \Iso ^{+} M$$
We already know that the  underlying topological space of $M/H$ is the 3-sphere and that the singular set is the theta-curve represented in Figure \ref{fig8}, up to knottings.
Now the point is to understand how $G/H$ acts on $M/H$. Since it is clear that two fixed point sets that have different orders of singularity cannot be identified, the possible actions of $G/H$  are the three represented in Figure \ref{fig11}.

\begin{figure}[h]
\begin{center}
\includegraphics[width=12cm]{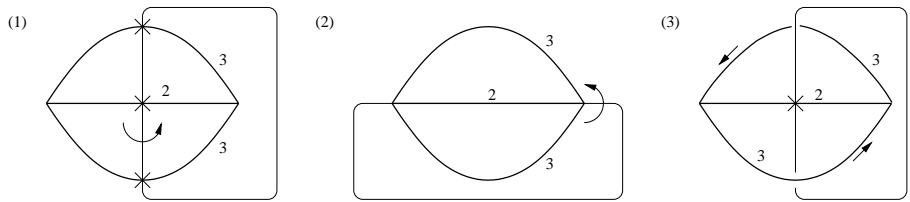}
\caption{Possible actions of order two on the singular set of $M/H$.}
\label{fig11}
\end{center}
\end{figure}

Nevertheless in this case we can exclude some actions.
For example the second action would produce as quotient a theta-curve that has one edge of singularity index four, and hence  we would have an element $\alpha \in \mathcal{S}_{4}$ of order four and  with non-empty fixed point set.
This would mean that $\alpha ^{2}$ is a hyperelliptic involution of $M$, then $\alpha$ projects to a symmetry that fixes pointwise  a 3-bridge and   by the positive solution of Smith  Conjecture, this is impossible.

In this case we have a reason to reject also the third action.
Notice that on $M/\langle t_{1},t_{2} \rangle$ acts also the dihedral group $\mathbb{D}_{6}\cong G/\langle t_{1},t_{2} \rangle$.
The dihedral group with six elements is generated by a transformation of order two and by a transformation of order three.
Since the transformation of order three acts on $M/\langle t_{1},t_{2} \rangle$, it must act also on its singular set, which is a theta-curve.
Therefore, as we have already seen, the fixed point set of this transformation must be non-empty.
Since the relation between the transformation of order two and the rotation of order three is dihedral in $\mathbb{D}_{6}$, the fixed point set of the involution is non-empty too and  the involution acts as a strong inversion on the fixed point set of the rotation of order three (see Proposition \ref{prop:normalizer}).
This means that the two fixed point sets intersect and this implies that the result of the action of $G/\langle t_{1},t_{2} \rangle$  on  $S^3$ produces as singular set of the quotient a theta curve with two edges of singularity index two and one edge of singularity index three. This theta curve must be contained in the singular graph of $M/G$, but this does not happen
for the "pince-nez" graph that we would obtain as singular set of the third action (while the tetrahedral graph resulting from the first case contains such a graph).
Therefore both the second and the third actions are not admissible.

Finally the only possible combinatorial setting of the singular set of the orbifold $M/G$ is the tetrahedral graph shown in Figure~\ref{fig12}.

\begin{figure}[htb]
\begin{center}
\includegraphics[width=3cm]{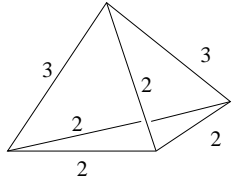}
\caption{Singular set of $M/G$, when $G \cong \mathcal{S}_{4}$.}
\label{fig12}
\end{center}
\end{figure}

\medskip

\noindent
\textbf{Case 2.6: $G \cong \mathbb{Z}_{2} \times \mathcal{S}_{4}$.}
\smallskip

We consider  the following subnormal series of subgroups of $G$ $$ \langle t_{1} \rangle \lhd \langle t_{1},t_{2} \rangle \lhd A \lhd J \lhd \Iso ^{+} M, $$
where $A$, as before, is the subgroup of $G$ isomorphic to $\mathbb{Z}_{2} \times \mathbb{Z}_{2} \times \mathbb{Z}_{2}$ obtained extending $\langle t_{1},t_{2} \rangle$ by an element that belongs to the centre of the group $\Iso ^{+} M$ and $A$ is the normal subgroup of $G$ isomorphic to $\mathbb{Z}_{2} \times \mathbb{A}_{4}$ and containing $A$ as normal subgroup. 
We already know that the  underlying topological space of $M/J$ is $S^{3}$ and that its singular set is the "pince-nez" graph represented in Figure~\ref{fig10} (always up to knottings).

It is clear that the action of $G/J$ on $M/J$ is a transformation with period two, since $\mathbb{Z}_{2} \times \mathbb{A}_{4}$ has index two in $\mathbb{Z}_{2} \times \mathcal{S}_{4}$.
The peculiarity of this case is that the singular set of the quotient is different according to the knotting of the "pince-nez" graph.
In fact the action on the graph is combinatorially unique, but the result depends on  the order of intersections of the two loops with the axis of the involution.
The action with period two on these graphs is always a rotation that fixes pointwise the middle edge and leaves invariant the two loops, operating a reflection on each of them, but we have to distinguish between the two cases represented by the graphs of Figure~\ref{fig14}. In the first case one of the two arcs, in which is divided the axis by the the intersection points between the axis and the first loop, does not contain any intersection point of the second loop and the axis; in the second case both the arcs contain an intersection point of the second loop and the axis.
\begin{figure}[t]
\begin{center}
\includegraphics[width=10cm]{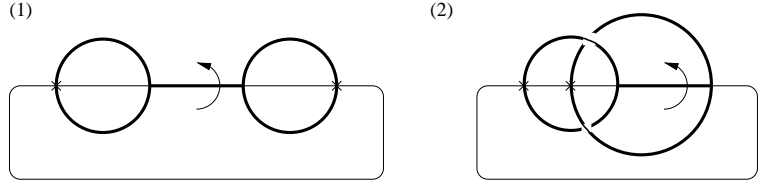}
\caption{Reflections on the "pince-nez" graphs.}
\label{fig14}
\end{center}
\end{figure}

We obtain two possible singular sets for the orbifold $M/G$, which are shown, up to knottings,  in Figure~\ref{fig15}.
In any case the underlying topological space is $S^3.$

In the table in Figure~\ref{final-table} we summarize the situation.

\begin{figure}[h!]
\begin{center}
\includegraphics[width=6cm]{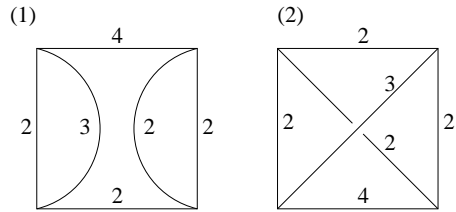}
\caption{Admissible singular sets for $M/G$, when $G\cong \mathbb{Z}_{2} \times \mathcal{S}_{4}$.}
\label{fig15}
\end{center}
\end{figure}

\begin{figure}[h!]
\begin{center}
\includegraphics[width=13cm]{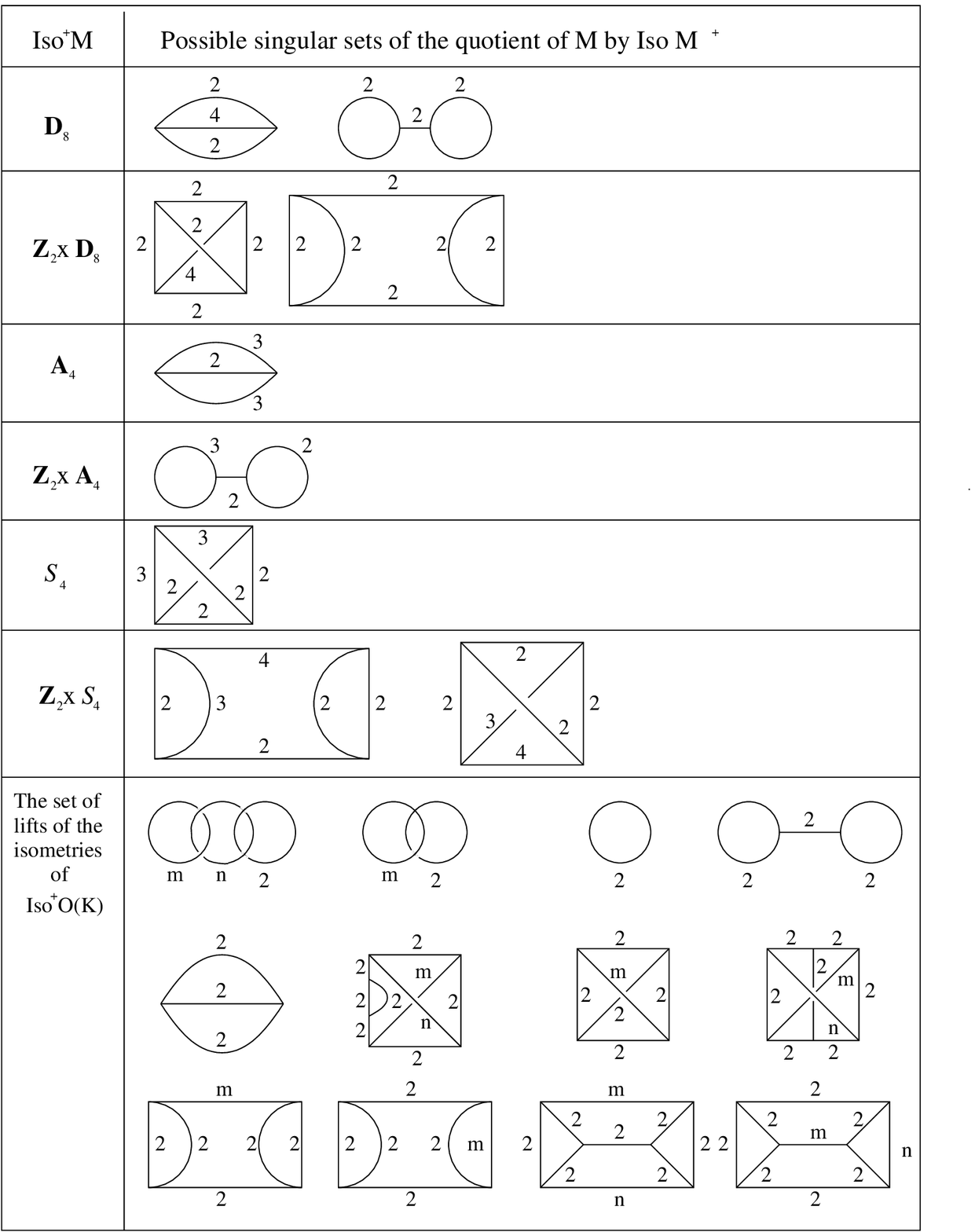}
\caption{Admissible singular set for $M/G$.}
\label{final-table}
\end{center}
\end{figure}

\section{The link case}\label{link}

In this section our aim is to generalize the work done in Chapter 3 on 2-fold branched covers of $\pi$-hyperbolic 3-bridge knots extending our considerations to 2-fold branched covers of $\pi$-hyperbolic 3-bridge links.
In light of the definition of bridge number we can deduce that 3-bridge links can have a maximum of three components.
Moreover, in contrast with the case of hyperbolic knots, the constituent knots of a hyperbolic link can also be all trivial.

We denote by $t$ the hyperelliptic involution that is the covering transformation of $L.$
In the last part of the  proof of \cite[Theorem 1]{R} it is proved that, if $L$ has more than one component, then  $t$ is central in $G.$ 
Therefore  we have that $M/G\cong (M/\langle t \rangle)/(G/\langle t \rangle)$.
We recall that the symmetry group of $L$ preserving the orientation of $S^3$ lifts to a finite group  acting on $M$ and containing $t$ in its center; 	by Thurston orbifold geometrization Theorem we can suppose that it is  contained up to conjugacy in $G$.

In this case, since  a long list of graphs would be produced (with respect to the one of the previous chapter), we don't consider the singular set of the quotients.
Our only aim this time is to analyze what the underlying topological space of this quotient is.

If $L$ is a link,  the isometry group of $M/\langle t \rangle$ is not as simple as when $L$ is a knot: it is no more true that it is  a subgroup of a dihedral group.
Let  $G_0$ be the normal subgroup of $G$ which consists of the elements fixing setwise each component of $\Fix t.$
We denote by $\bar G$ (resp. $\bar G_0$) the quotient $G/ \langle t \rangle$ (resp. $G_0/ \langle t \rangle$); the group $\bar G_0$ fixes setwise each component of $L.$

Clearly we have that the quotient group $\bar G/ \bar G_0$ is a subgroup of the symmetry group $\mathcal{S}_{n}$, where $n$ is the number of components of the link $L$, hence  in our case  $\bar G/ \bar G_0$  is either a subgroup of $\mathcal{S}_{3}$ or a subgroup of $\mathcal{S}_{2}$.
By Proposition~\ref{prop:normalizer},  $\bar G_0 \leq \mathbb{Z}_{2}( \mathbb{Z}_{m} \times 	\mathbb{Z}_{n} )$ for some $m,n \in \mathbb{N}$, i.e. $\bar G_0$ is isomorphic to a subgroup of a generalized dihedral group. Therefore $\bar G_0$ has an abelian subgroup of rank at most two of index at most two.
We separately analyze   the different cases.

\medskip

\noindent
\textbf{Case 1: $\bar G_0 \cong \mathbb{Z}_{2} ( \mathbb{Z}_{m} \times \mathbb{Z}_{n} )$,  i.e. $\bar G_0$ is generalized dihedral.}

We recall that the underlying topological space of $M/\langle t \rangle$ is $S^3.$ 
By  Lemma~\ref{lemma:orth-group}, we obtain that the underlying topological space of the orbifold  $(M/\langle t \rangle)/\bar G_0\cong M/G_0$  is $S^3$.
What remains to study now is the action of $G/G_0$ on $M/G_0$, but $G/G_0$ is either a  subgroup of $S_3$ or a subgroup of $S_2$, and hence it is cyclic or dihedral.
By Lemma~\ref{lemma:orth-group}, we obtain that the underlying topological space of $M/G\cong (M/G_0)/(G/G_0)$ is either $S^3$ or a  Lens space.

\medskip
\noindent
\textbf{Case 2: $\bar G_0 \leq ( \mathbb{Z}_{m} \times \mathbb{Z}_{n} )$, i.e. $\bar G_0$ is abelian.}

In this case what is missing with respect to the previous one is the action of a strong inversion on a component  of $L$.
Again two cases can occur: $\rank \bar G_0=1$ or $\rank \bar G_0=2$.

\smallskip
\noindent
\textbf{Case 2.1: $\rank \bar G_0=2$.}

In this case $\bar G_0$ admits a subgroup either of type $\mathbb{Z}_{p} \times \mathbb{Z}_{p}$ with $p$ an odd prime or of type $\mathbb{Z}_{2} \times \mathbb{Z}_{2}$.

We begin showing that the first case cannot occur.
Suppose that there exists a subgroup $D$ of $\bar G_0$  such that $D \cong \mathbb{Z}_{p} \times \mathbb{Z}_{p}$ for some prime $p > 2$.
Let $L_{i}$ be a connected component of the link $L$ and let $X_{i}$ be the subgroup of $D$ made of the isometries that fix pointwise the component $L_{i}$.
The group  $D/X_{i}$  acts faithfully  on the component $L_{i}$,  hence the group $D/X_i$ must be either cyclic or dihedral. Clearly it cannot be dihedral, being a quotient of an abelian group. This means that $\frac{D}{X_{i}}$ is cyclic, in particular isomorphic to $\mathbb{Z}_{p}$.
Therefore, by Lemma~\ref{lemma:orth-group} the group $X_i$ is one of the two subgroups of $D$ isomorphic to $\mathbb{Z}_{p}$ that admit nonempty fixed-point set.
This argument holds true for all the components of $L$:  the components of $L$ are two and, since $D$ can be simultaneously conjugate to block-diagonal matrices (see proof of Lemma~\ref{lemma:orth-group}), $L$ is the Hopf Link. 
Since  the Hopf link is a well known 2-bridge link, this leads to a contradiction.

Suppose now that $\bar G_0$  contains a subgroup $D$ isomorphic to $\mathbb{Z}_{2} \times \mathbb{Z}_{2}$, in this case we prove that the underlying topological space of $M/G_0$ is $S^3$.
Again we consider the quotient $D/X_{i}$, where $X_{i}$ is the normal subgroup of $D$ consisting of the isometries that fix pointwise the component $L_{i}$ of $L$.
In this case $D/X_{i}$ can be either cyclic or the whole group $D$.

We suppose that $X_i$ is trivial for some component $L_i$, thus $D$ contains an $L_i$-reflection which we denote by $\alpha$.
Suppose that exists an element $\beta \in \bar G_0$ of order different from two.
This element must  act as a rotation on the $i$-th component of $L$, but then, since $\alpha$ is a $L_{i}$-reflection, we have that $\alpha\beta\alpha^{-1}=\beta^{-1}.$ 
This implies that $\bar G_0$ admits a dihedral subgroup, that leads to a contradiction, being $\bar G_0$ abelian.
Therefore we obtain that  $\bar G_0$ is isomorphic to $\mathbb{Z}_{2} \times \mathbb{Z}_{2}$.
In this case we are done.
In fact, since $\mathbb{Z}_{2} \times \mathbb{Z}_{2}$ is always generated by a couple of involutions with non empty fixed point set, then the underlying topological space of $M/G_0\cong (M/\langle t \rangle)/(G_0/\langle t \rangle)$ is $S^{3}.$

On the other hand, if $X_{i}$ is non-trivial for each component $L_{i}$ of the link $L$, then again we obtain a contradiction by Lemma~\ref{lemma:orth-group}.
In fact, if $D$ contains two involutions with non empty fixed point set, then $L$ should be the Hopf link. 
If $D$ contains three involutions with non empty fixed point set, the components of $L$ should intersect in two points and this is  impossible.

Summarizing we obtained that if $\bar G_0$ has rank two, then $M/G_0$ has always underlying topological space $S^{3}$.
As in the previous case, the group $G/G_0$ is either cyclic  or isomorphic to $\mathcal{S}_{3}$, and this implies that the underlying topological space of $M/G$ is either $S^{3}$ or a lens space.

\smallskip
\noindent
\textbf{Case 2.2: $\rank \bar G_0=1$, i.e. $\bar G_0$ is cyclic.}

The quotient of $S^{3}$ by a cyclic group of isometries is an orbifold with underlying topological space either $S^{3}$ or a lens space.

We distinguish two cases: $\bar G_0$  admits at least one element acting with non empty fixed point set or $\bar G_0$ acts freely.

If $\bar G_0$ does not act freely, then each element in the normalizer of $G_0$  fixes setwise each  curve fixed by a nontrivial element of $G_0$.
In fact the different curves are fixed pointwise by   elements of different order.
This means that $G$ fixes setwise at least a closed curve, therefore, thanks to Proposition~\ref{prop:normalizer}, we can say that $G$ must be a subgroup of a generalized dihedral group  and we are done by Lemma~\ref{lemma:orth-group}.

On the other hand if $\bar G_0$ acts freely, then the analysis of the quotient $M/G$ is more complicated.
If $\bar G= \bar G_0$, then the underlying topological space is a lens space.

Otherwise the quotient $G/G_0$ is isomorphic  to $\mathbb{Z}_{2}$, $\mathbb{Z}_{3}$ or $\mathcal{S}_{3}$.

If  $\bar G/ \bar G_0 \cong \mathbb{Z}_{3}$, then by Lemma~\ref{lemma:orth-group} the group $\bar G$ is abelian and the underlying topological space of $M/G$ is either  $S^3$ or a lens space. If the components are three and $\bar G/ \bar G_0 \cong \mathbb{Z}_{2}$, then one of the component of $L$ is fixed setwise by $\bar G$ and we are done.

In the remaining cases we  can suppose that $\bar G$ has an abelian subgroup of index two.
Up to now we were able to prove that the underlying topological space of $M/G$ is either $S^3$ or a Lens space, unfortunately in the remaining cases some groups can admit as underlying topological space of the quotient a prism manifold. By analyzing the remaining groups case by case, we can deduce more information about the situations in which a prism manifold can occur, but at this point we prefer to give a shorter argument that simply exclude tetrahedral, octahedral and  icosahedral manifolds as underlying topological space of the quotient.
Since $\bar G$ has an abelian subgroup of index two,  the group leaves invariant a fibration of $S^3$ (see~\cite{MS}). The quotient orbifold $S^3/\bar G$ admits a Seifert fibration induced by the fibration of $S^3$ left invariant by $G.$ By~\cite[Lemma 2]{MS}, the base 2-orbifold $B$ of $S^3/\bar G$ is the  quotient of $S^2$ by the action of $\bar G$ (which is possibly non-faithfully.)  Since $G$ has an abelian subgroup of index 2,  either $B$ has a disk as underlying topological space or it is a 2-sphere with at most one singular point of index strictly greater then 2 (i.e. the cases with base 2-orbifold $S^2(2,3,3)$ $S^2(2,3,4)$ and $S^2(2,3,5)$ in ~\cite[Table 4]{MS} are excluded).

If $B$ has underlying topological space the 2-disk, then, by~\cite[Proposition 2.11]{Dun}, the underlying topological space of $S^3/\bar G$ (and hence of $M/G$) is either $S^3$ or a Lens space. 

On the other hand, if the base 2-orbifold has no boundary component, then, by forgetting the orbifold singularity of the  fibers, we obtain, from the Seifert fibration of $S^3/\bar G$,   a Seifert fibration of the underlying topological space of $S^3/\bar G.$ The base 2-orbifold of the underlying topological space of $S^3/\bar G$ can be obtained from $B$ by dividing the index of the singular points by the singularity index of the corresponding fibers. The Euler number of the fibration is not affected by the singularity forgetting process. Since $S^2(2,3,3)$ $S^2(2,3,4)$ and $S^2(2,3,5)$ are excluded as base 2-orbifolds, it turns out that the underlying topological space of $S^3/\bar G$ is $S^3$, a lens space or a Prism manifold (see~\cite{McC} and \cite[Table 2,3,4]{MS})

\section{An example}

In this section we describe an infinite family of hyperbolic 3-manifolds such that each of them is the 2-fold branched covering  of three inequivalent knots, two of them with bridge number equal to three and the third one with bridge number equal or greater than four.

For any triple of nonzero integers $(i,j,k)$ we can define the 3-bridge knot
$K_{ijk}$ presented in Figure~\ref{3bridge}.

\begin{figure}[h!]
\begin{center}
\includegraphics[width=5cm]{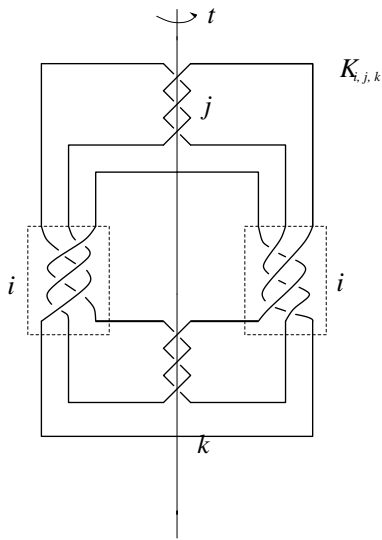}
\caption{The knot $K_{ijk}$.}
\label{3bridge}
\end{center}
\end{figure}

In  Figure~\ref{3bridge} we have also drawn the axis of a strong inversion $t$ of $K_{ijk}$. Let $\mathcal{O} (K_{ijk})$  be the orbifold with underlying topological space $S^3$ and $K_{ijk}$ as singular set of index 2. 
We consider the quotient orbifold $\mathcal{O}(\theta_{ijk}):= \mathcal{O} (K_{ijk})/\langle t \rangle$ which has $S^3$ as underlying topological space. The  singular set is a theta-curve
$\theta_{ijk}$ with edges $e_1$, $e_2$ and $e_3$ and constituent knots
$A_1=e_2\cup e_3$, $A_2=e_1\cup
e_3$ and $A_3=e_1\cup e_2$. The theta curve $\theta_{ijk}$ is represented in  Figure~\ref{theta} (how to obtain this planar diagram of $\theta_{ijk}$ is explained in  \cite{MR}).

\begin{figure}[h!]
\begin{center}
\includegraphics[width=5cm]{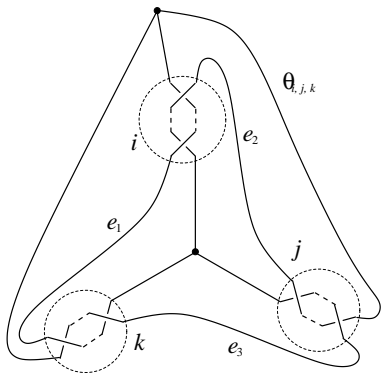}
\caption{The theta curve $\theta_{ijk}$.}
\label{theta}
\end{center}
\end{figure}

The three constituent knots
are trivial and the preimage of
$e_1$,
$e_2$ respectively
$e_3$ in the 2-fold cyclic branched covering of $S^3$ along $A_1$, $A_2$
respectively  $A_3$ is
$K_1=K_{ijk}$, $K_2=K_{jki}$ respectively $K_3=K_{kij}$.
Finally, if we take the two fold branched covering of $K_1$, $K_2$ or $K_3,$
we get the same manifold $M$: the
manifold $M$ is the $\mathbb{D}_4$ covering of $\mathcal{O} (\theta_{ijk})$.
In \cite{Z} it is proved that  $M$ is hyperbolic for $|i|,|j|,|k|$ sufficiently large. 
The isometry group of $M$  was studied in \cite{MR}. Here we consider  the case where two of the indices are equal, while the third one is different. If $\{i,j,k\}$ is not of the form  $\{lm,(l-1)m\}$, with $m$ and $l$ integers and $l$ even, then the isometry group of $M$ is isomorphic to  $\mathbb{D}_8$ and it does not contain any orientation-reversing isometry (see \cite[page 8]{MR}). We denote by $G$ the isometry group of $M$ and we suppose $i=j$.
In $G$ there are three conjugacy classes of involutions, two of them consist of hyperelliptic involutions and correspond to the 3-bridge knots. One of the hyperelliptic involutions is central in $G$. The involutions in the third conjugacy class are not hyperelliptic. In this section we prove that the quotient orbifold of $M$ by one of these involutions has underlying topological space $S^3$; the singular set of this orbifold  is a knot that has bridge number different from three. 
If $i=j$,  the theta curve $\theta_{iik}$ has a symmetry of order 2 exchanging the two vertices and leaving setwise invariant only one of the edges. In the diagram of $\theta_{iik}$ presented in Figure~\ref{theta2} the symmetry is evident, it consists of a $\pi$-rotation around the point $C$. We denote this involution by $\tau$.

\begin{figure}[h!]
\begin{center}
\includegraphics[width=6cm]{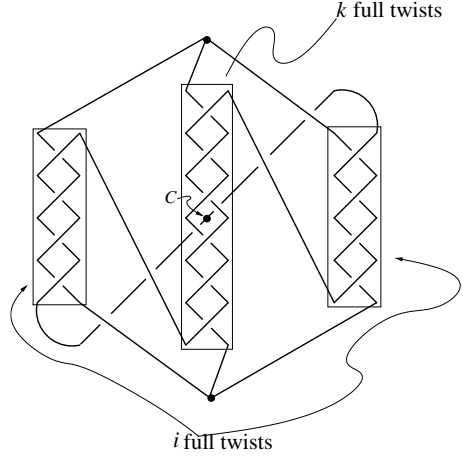}
\caption{Symmetry of $\theta_{iik}$.}
\label{theta2}
\end{center}
\end{figure}

The quotient of $\mathcal{O}(\theta_{iik})$ by $\tau$ is an orbifold with $S^3$ as underlying topological space and with the knotted pince-nez  graph represented in Figure~\ref{quotient} as singular set.  This orbifold is $M/G$ (this situation corresponds to  Case 2.1. in Section~\ref{knot}).   We denote by $l_0$ and $l_1$ the  loops of the pince-nez graph. In particular let $l_1$ be the projection of the axis of $\tau$ (the dotted line in Figure~\ref{quotient}).  

\begin{figure}[h!]
\begin{center}
\includegraphics[width=6cm]{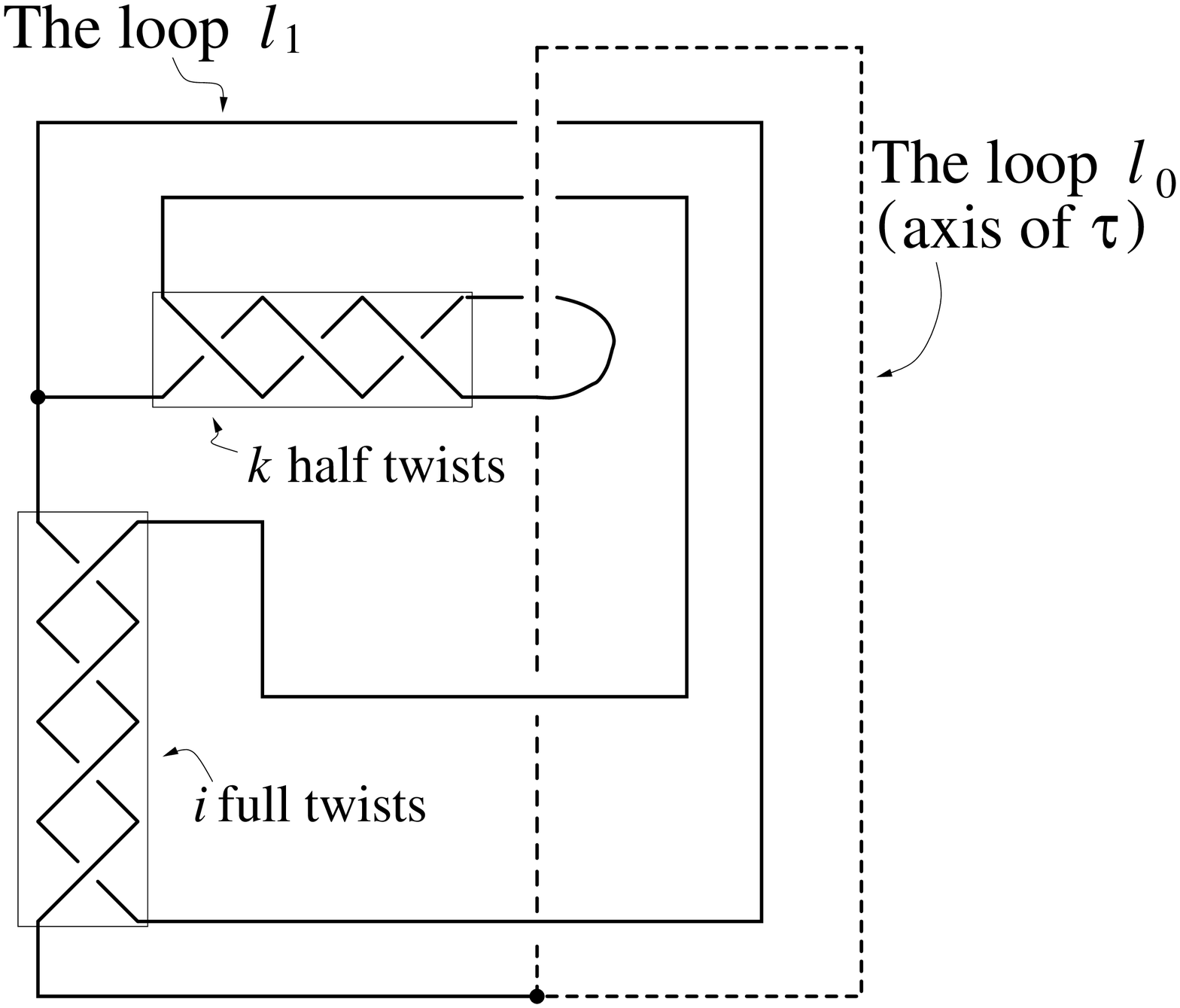}
\caption{The singular set of $M/G$.}
\label{quotient}
\end{center}
\end{figure}

Now we consider the orbifold obtained by taking the 2-fold covering of $M/G$ branched over the loop $l_1$. We remark that the loop $l_1$ is a trivial knot.  This gives an orbifold $\mathcal{O}(\Gamma_{ik})$ with $S^3$ as underlying topological space and the theta curve $\Gamma_{ik}$ represented in Figure~\ref{final-theta} as singular set (of singularity index 2). 

\begin{figure}[h!]
\begin{center}
\includegraphics[width=6cm]{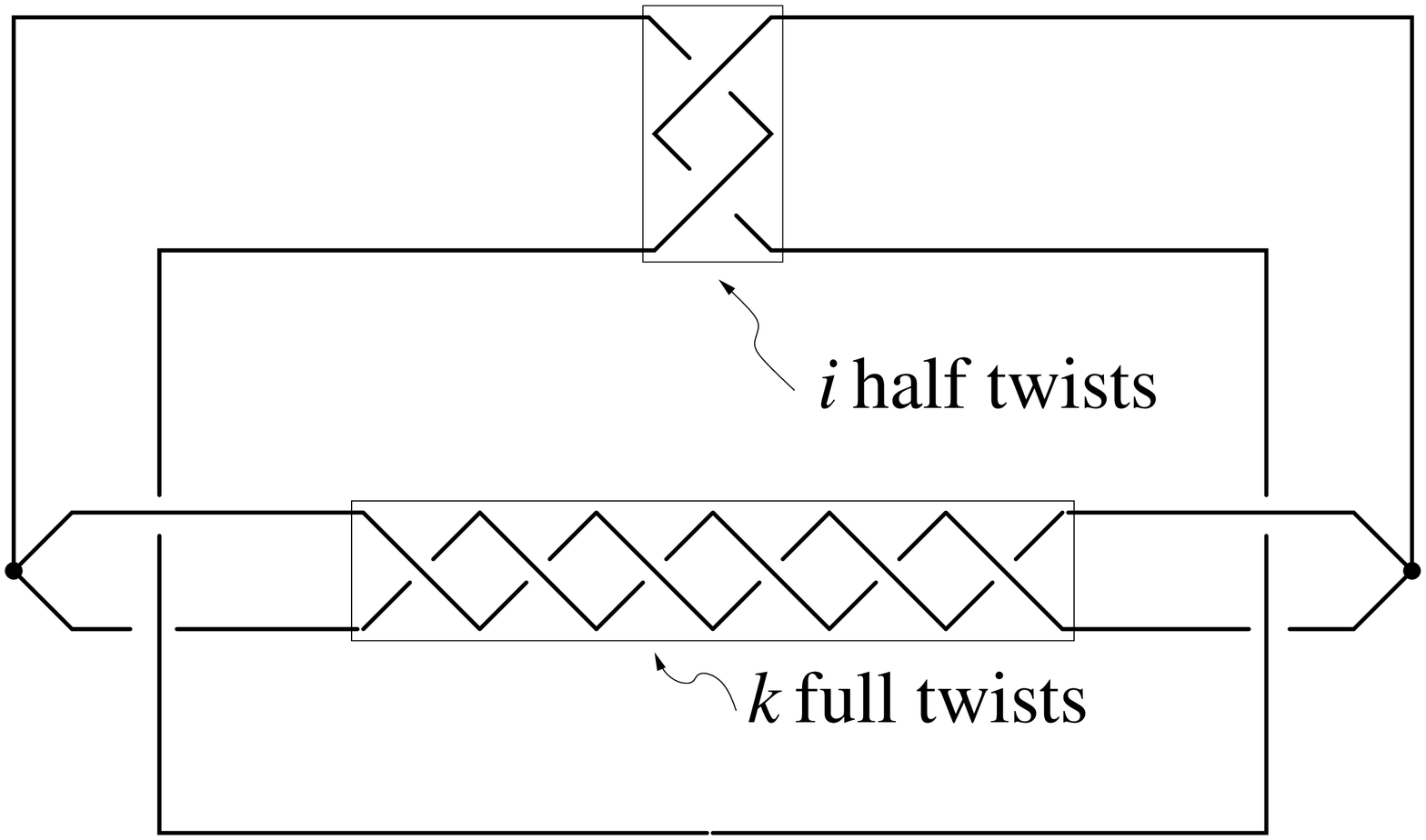}
\caption{$\Gamma_{iik}$}
\label{final-theta}
\end{center}
\end{figure}

To draw explicitly  $\Gamma_{ik}$ we use a planar diagram of the singular set  of $M/G$ where $l_1$ has a trivial projection. In Figure~\ref{transformation}  how to obtain such a diagram is explained. From this representation of the graph it is easy to reconstruct a diagram of $\Gamma_{ik}$.  To help the reader, in Figure~\ref{final-theta} we represent explicitly the axis of the involution acting on $\mathcal{O}(\Gamma_{ik})$ that gives $M/G$ as quotient.

\begin{figure}[h!]
\begin{center}
\includegraphics[width=10cm]{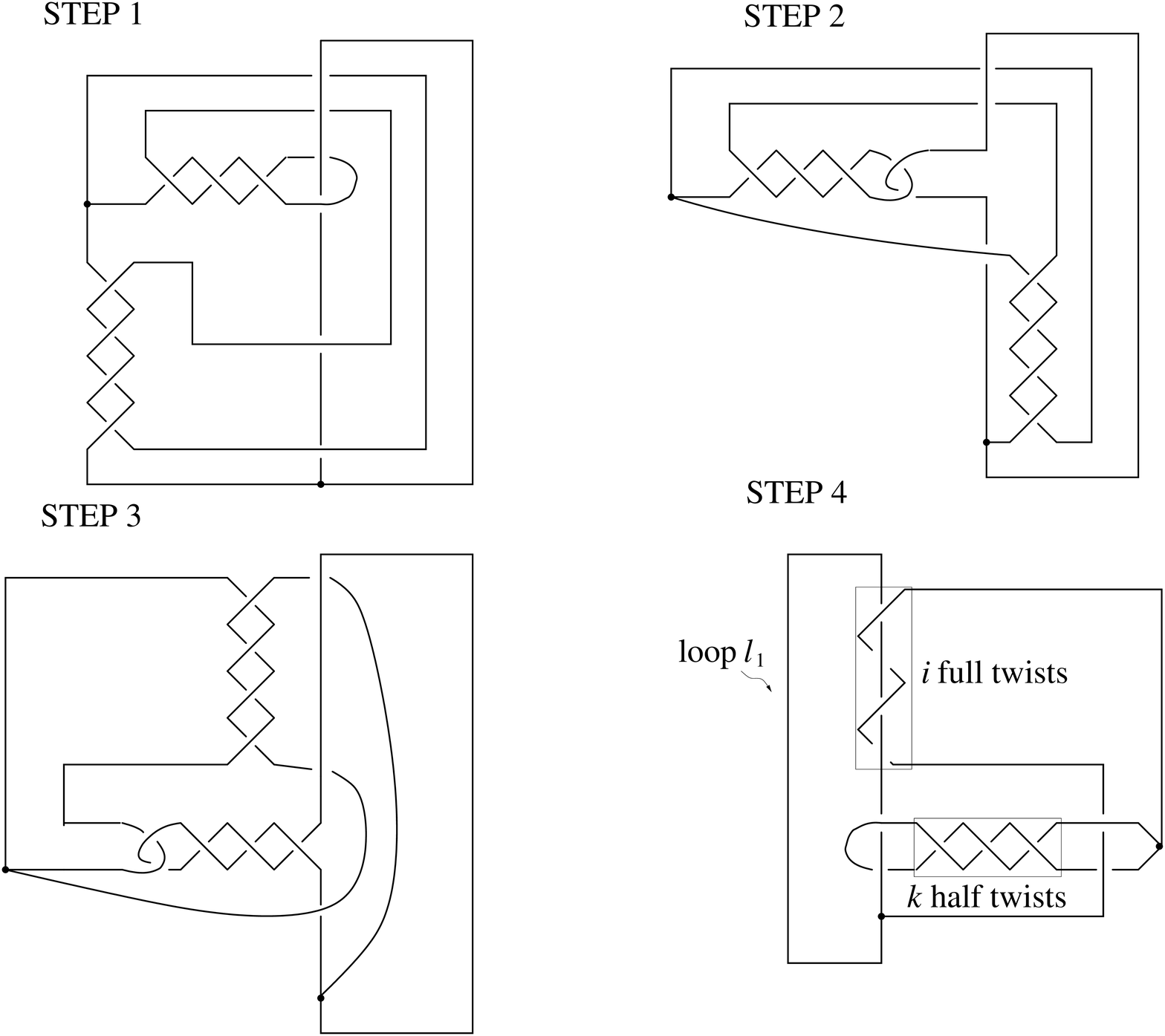}
\caption{$\Gamma_{iik}$}
\label{transformation}
\end{center}
\end{figure}

\begin{figure}[h!]
\begin{center}
\includegraphics[width=6cm]{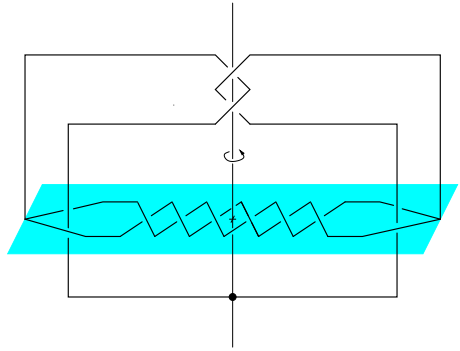}
\caption{$\Gamma_{iik}$}
\label{final-theta-bis.eps}
\end{center}
\end{figure}

The orbifold $\mathcal{O}(\Gamma_{ik})$ is the quotient of $M$ by the group generated by the central hyperelliptic involution  and by one involution that is not hyperelliptic (we call $u$ such an involution). The orbifold  $M/\langle u \rangle$ is the two fold  covering of $\mathcal{O}(\Gamma_{ik})$ branched over the appropriate constituent knot. We remark that all the three constituent knots of $\Gamma_{ik}$ are trivial.  This implies that the underlying topological space of $M/\langle u \rangle$ is $S^3$ and that $M$ is the 2-fold  covering of $S^3$ branched over  the knot that is the singular set of $M/\langle u \rangle.$
We denote this knot by $L_{ik}.$
Since $u$ is not hyperelliptic, the bridge number of $L_{ik}$ is different from 3. Since the only knot with bridge number one is the unknot and the 2-fold branched coverings of the 2-bridge knots are the lens spaces, these knots have bridge number greater or equal then four. 
We give a diagram of  $L_{ik}$ in Figure~\ref{new-knot-final} and Figure~\ref{new-knot} shows a procedure to get it. Indeed an explicit diagram of $L_{ik}$ is not necessary to get the properties we need.
We remark that the  Heegaard splitting induced by a minimal bridge presentation of  $L_{ik}$ is not of minimal genus.

\begin{figure}[h!]
\begin{center}
\includegraphics[width=5.5cm]{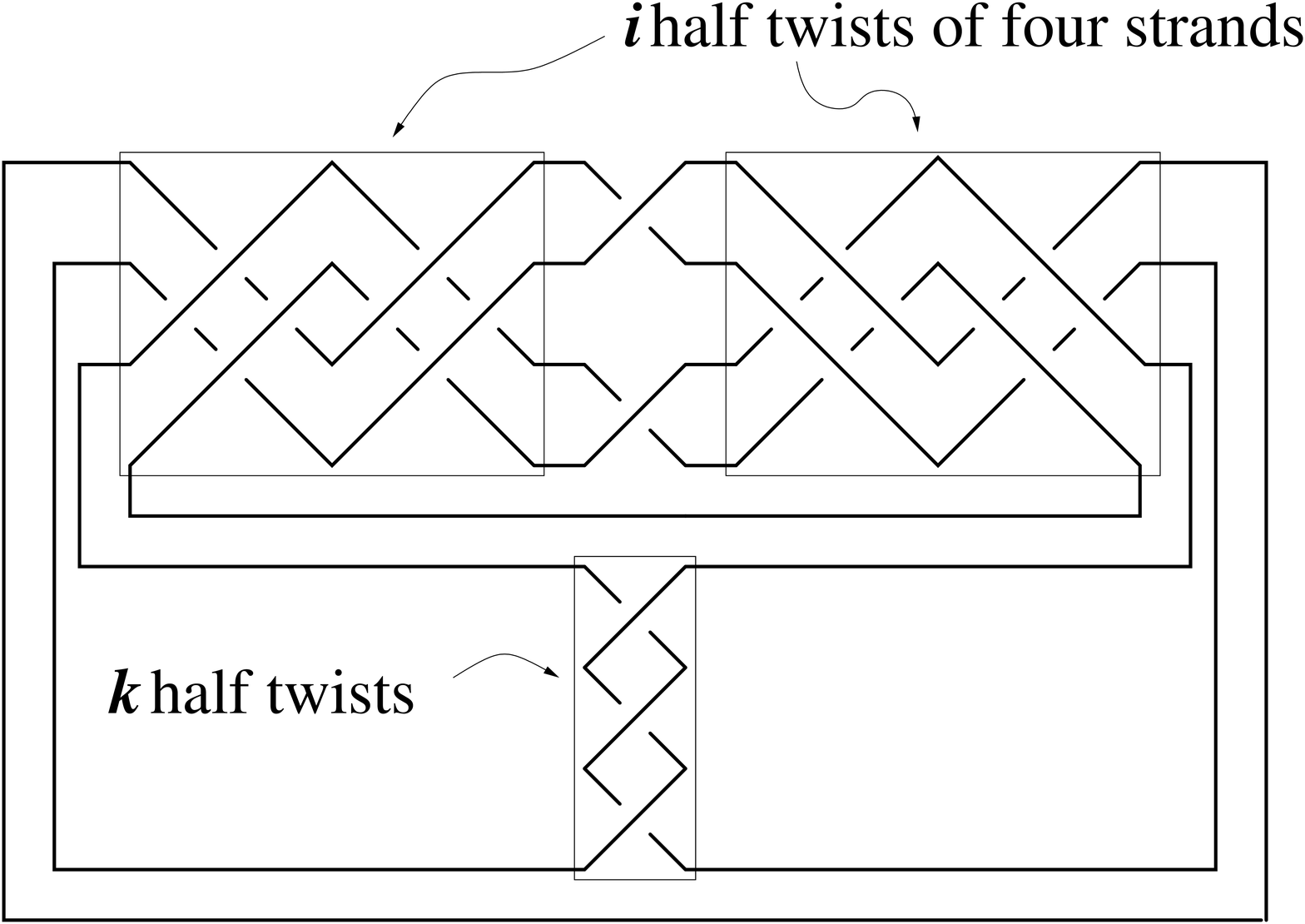}
\caption{the knot $L_{ik}$}
\label{new-knot-final}
\end{center}
\end{figure}

\begin{figure}[h!]
\begin{center}
\includegraphics[width=8.5cm]{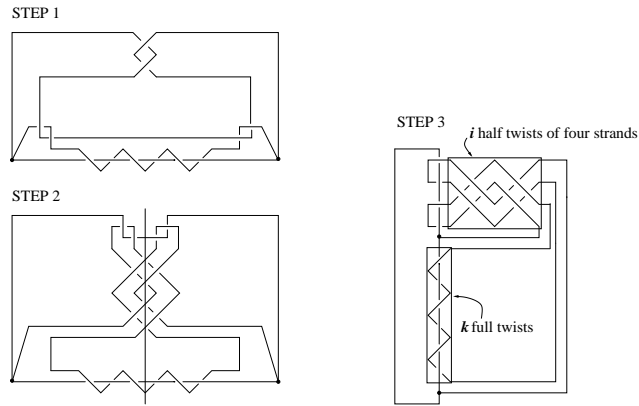}
\caption{Equivalent diagrams of $G_{ik}$}
\label{new-knot}
\end{center}
\end{figure}

\end{document}